\documentclass{amsart}

\pdfoutput=1
\usepackage[pdftex]{graphics}
\usepackage{amsmath,amssymb,mathtools}
\usepackage{graphicx}
\usepackage{xcolor}
\usepackage{epsdice}
\usepackage{enumitem}
\usepackage{mathtools}
\usepackage[colorlinks]{hyperref}

\graphicspath{{./Figures/}}

\newtheorem{theorem}{Theorem}[section]

\theoremstyle{definition}

\newtheorem{conjecture}[theorem]{Conjecture}

\newtheorem{lemma}[theorem]{Lemma}

\theoremstyle{remark}
\newtheorem{remark}[theorem]{Remark}

\begin{document}
%\preprint{IPMU17-0051}

\author[Dongmin Gang]{Dongmin Gang}

\author[Seonhwa Kim]{Seonhwa Kim}

\author[Seokbeom Yoon]{Seokbeom Yoon}

\address{(DG) Asia Pacific Center for Theoretical Physics (APCTP), Headquarters San 31, Hyoja-dong, Nam-gu, Pohang 790-784, Korea}
\address{(SK and SY) School of Mathematics, Korea Institute for Advanced Study, Seoul 02455, Korea}
\title[Torsion]
{Adjoint Reidemeister torsions from wrapped M5-branes}

\begin{abstract}
We introduce a vanishing property of adjoint Reidemeister torsions of a cusped hyperbolic 3-manifold derived from the physics of wrapped M5-branes on the manifold. To support our physical observation, we present a rigorous proof for the figure-eight knot complement with respect to all slopes. We also present numerical verification for several knots.
\end{abstract}

%%%%%%%%%%%%%%%%%%%%%%%%%%%%%%%%%%%%%%%%%%%%%%%

\maketitle
\tableofcontents

%%%%%%%%%%%%%%%%%%%%%%%%%%%%%%%%%%%%%%%%%%%%%%
\section{Introduction}

 Since the advent of Witten's reformulation of the Jones polynomial 
 using $\mathrm{SU}(2)$ Chern-Simons theory \cite{witten1989}, 
 there have been fruitful interplays between 
 3-dimensional quantum field theories (3D QFTs) and mathematics of 3-manifolds and knots.
 In the interplays, mathematicians provide rigorous approaches to topological quantum field theories while physicists suggest new topological invariants and conjectures which seem to be unexpected to mathematicians. 
  Recently,  3D-3D correspondence \cite{Terashima:2011qi,Dimofte:2011ju}   accelerates the interplays drastically as it allows us to study various supersymmetric quantities in 3D supersymmetric QFTs in terms of mathematical quantities of 3-manifolds.
 
 In this paper, we  study 3D $\mathcal{N}=2$ supersymmetric QFTs $\mathcal{T}_{N=2}[\mathcal{M},\mathcal{K}]$ labeled by a closed 3-manifold $\mathcal{M}$ and a knot $\mathcal{K}\subset \mathcal{M}$.
Under mild assumptions, we propose concrete conjectures on the knot exterior $M=\mathcal{M}\backslash \nu(\mathcal{K})$ deduced from the twisted indices, the partition functions of $\mathcal{T}_{N=2}[\mathcal{M},\mathcal{K}]$ placed on curved backgrounds $\Sigma_g \times S^1$. 
 Here $\nu(\mathcal{K})$ denotes a tubular neighborhood of $\mathcal{K}$ and $\Sigma_g$ is a Riemann surface of genus $g$.
 We give a particular focus on the conjecture for $g=0$ (see Conjecture \ref{conj:intro} below) which is most unexpected in the mathematical viewpoint.
 
 \begin{conjecture} \label{conj:intro} Let $M$ be a compact 3-manifold with a torus boundary whose interior admits a hyperbolic structure. Let $X^\mathrm{irr}(M)$ be the character variety of irreducible $\mathrm{SL}_2(\mathbb{C})$-representations.
 	Suppose that every irreducible component of $X^\mathrm{irr}(M)$ is of dimension 1. Then for any slope $\gamma \in H_1(\partial M; \mathbb{Z})$ we have
 	\[\sum_{[\rho] \in \mathrm{tr}_\gamma^{-1}(z)} \frac{1} {\mathrm{Tor}(M;\mathfrak{g}_\rho,\gamma) } =0\]
 	for generic $z \in \mathbb{C}$. Here $\mathrm{tr}_\gamma : X^\mathrm{irr}(M)\rightarrow \mathbb{C}$ is the trace function of $\gamma$ and $\mathrm{Tor}(M;\mathfrak{g}_\rho,\gamma)$ is the adjoint Reidemeister torsion with respect to $\rho$ and $\gamma$.
 \end{conjecture}
 
 We observe that Conjecture \ref{conj:intro} is deeply related to global residue theorem. In particular, we give a rigorous proof of Conjecture \ref{conj:intro} for the figure-eight knot exterior. 
 Note that this gives an infinite family of $\mathcal{T}_{N=2}[\mathcal{M},\mathcal{K}]$ (obtained by varying a slope $\gamma$) where our physical arguments used to derive the conjecture are mathematically supported.
We also present numerical verification to provide non-trivial consistency checks for the assumptions in  Conjecture \ref{conj:intro}.
In particular, we emphasize
that all irreducible components of $X^\mathrm{irr}(M)$ (not just the geometric component) should be taken into account in the 3D-3D relation. Otherwise, the twisted index at $g=0$ may not be even an integer (see Section \ref{sec:knot74} and Remark \ref{rmk:allcomponent}).

 The paper is organized as follows.
 In Section \ref{sec:3D-3D}, we give a general survey on 3D-3D correspondence. 
 In Section \ref{sec: conjecture from physics}, we focus on the 3D-3D relation for the twisted indices which leads Conjecture \ref{conj:intro}. 	
 In Section \ref{sec:def}, we recall some definitions and properties of adjoint Reidemeister torsion. 
 In Section \ref{sec:proof}, we prove the conjecture for the figure-eight knot exterior.
 We discuss some further directions in Section \ref{sec:further}.

%%%%%%%%%%%%%%%%%%%%%%%%%%%%%%%%%%%%%%%%%%%%%%%%%%%

\section{A brief  survey on 3D-3D correspondence} \label{sec:3D-3D}
The 3D-3D correspondence relates 3D suppersymmetric quantum field theories (SQFTs) and mathematcis of 3-manifolds.
The correspondence can be understood from the physics of M5-branes, 6-dimensional extended objects in M-theory.
Precisely, we consider the low-energy world-volume theory of multiple $N$ M5-branes called the 6D $A_{N-1}$ $(2,0)$ theory.
\begin{align}
\begin{split}
&(\textrm{6D $A_{N-1}$ (2,0) theory}) 
\\
&=(\textrm{low-energy world-volume theory of coincident $N$ M5-branes}).
\end{split}
\nonumber
\end{align}
Through a twisted compactification of the 6D theory along a closed 3-manifold $\mathcal{M}$ with a defect along a knot $\mathcal{K} \subset \mathcal{M}$, we can geometrically engineer a 3D SQFT $\mathcal{T}_N[\mathcal{M},\mathcal{K}]$ labeled by a pair $(\mathcal{M},\mathcal{K})$ and $N\geq2$.
\begin{align}
\begin{split}
\mathcal{T}_N[\mathcal{M},\mathcal{K}]&:= (\textrm{3D theory obtained from a twisted compactification of }
\\
&\textrm{6D $A_{N-1}$ (2,0) theory along $\mathcal{M}$ with a regular maximal defect along $\mathcal{K}$}).
\end{split}
\nonumber
\end{align}
To preserve some supersymmetries, we perform a partial topological twisting using the usual $SO(3)$ subgroup of $SO(5)$ R-symmetry of the 6D theory.
Then the resulting 3D theory $\mathcal{T}_{N}[\mathcal{M},\mathcal{K}]$ has 3D $\mathcal{N}=2$ superconformal symmetry with $su(N)$ flavor symmetry associated to the maximal regular defect. An explicit field-theoretic construction of  $\mathcal{T}_N[\mathcal{M},\mathcal{K}]$ was proposed in \cite{Dimofte:2011ju,Dimofte:2013iv} using an ideal triangulation of the knot complement.
One interesting aspect of the construction is that partition functions (ptns) of $\mathcal{T}_N[\mathcal{M},\mathcal{K}]$ on supersymmetric (SUSY)  curved backgrounds $\mathbb{B}$ give topological invariants of $(\mathcal{M}, \mathcal{K})$.
\begin{align}
\begin{split}
\textrm{3D-3D relation : }  &(\textrm{SUSY ptns of $\mathcal{T}_{N}[\mathcal{M},\mathcal{K}]$ on $\mathbb{B}$})
\\
&= (\textrm{topological invariants $\mathcal{I}_N (\mathcal{M}, \mathcal{K},\mathbb{B})$}).
\end{split}
\nonumber
\end{align}

During the last decades, physicists have studied various supersymmetric backgrounds $\mathbb{B}$ and developed techniques for computing the SUSY ptn on $\mathbb{B}$ using so-called localization technique (see \cite{Willett:2016adv} for a review).  For $\mathbb{B} = S^2\times_q S^1$ the corresponding SUSY ptn is called a superconformal index \cite{Kim:2009wb} where $q$ denotes the Omega-deformation parameter on $S^2\times S^1$. 
For $\mathbb{B} = S^3_b/\mathbb{Z}_k$ the corresponding SUSY ptn is called a squashed Lens space partition function \cite{Hama:2011ea} where $b$ denotes the squashing (Omega-deformation) parameter. For the above cases, the 3D-3D relation is known as follows.
\begin{align}
\begin{split}
&(\textrm{$S^3_b/\mathbb{Z}_k$ partition function of $\mathcal{T}_{N}[\mathcal{M},\mathcal{K}]$}) 
\\
& \qquad =  (\textrm{State-integral model of $\mathrm{SL}_N(\mathbb{C})$ theory of level $k$ \cite{Dimofte:2014zga}});
\\
&(\textrm{Superconformal index of $\mathcal{T}_{N}[\mathcal{M},\mathcal{K}]$})
\\
& \qquad =  (\textrm{3D index \cite{Dimofte:2011py,Garoufalidis:2016ckn,Dimofte:2013iv}} ).
\end{split}
\nonumber
\end{align}

Another interesting SUSY partition function which has not been explored seriously  in the context of 3D-3D correspondence until quite recent days  is a twisted index \cite{Benini:2015noa,Benini:2016hjo,Closset:2016arn,Gukov:2015sna,Gukov:2016gkn}, the SUSY ptn on  $ \mathbb{B}=\Sigma_g\times S^1$. Here $\Sigma_g$ is a closed Riemann surface of genus $g$. Let
\begin{align}
\mathcal{I}_N (\vec{x};\mathcal{M}, \mathcal{K}, g )= (\textrm{Twisted index of $\mathcal{T}_{N}[\mathcal{M},\mathcal{K}]$ on $\Sigma_g \times S^1$}) \label{twistedindex}
\end{align}
where $\vec{x}  = \{ x_i\}_{i=1}^{N-1}$  are fugacities for $(N-1)$ Cartan generators of the $su(N)$ flavor symmetry. 
When $N=2$, the precise 3D-3D relation  for the twisted index is given as follow \cite{Gang:2018hjd,Gang:2019uay}:
\begin{align}
\mathcal{I}_{N=2}(x;\mathcal{M},\mathcal{K},g) =  \mkern-20mu \sum_{[\rho] \in\mathrm{tr}_\gamma^{-1}(x+x^{-1}) }  \mkern-30mu \left( d_\gamma \cdot \textrm{Tor}(M;\mathfrak{g}_\rho, \gamma) \right)^{g-1}
\label{3d-3d relation}
\end{align}
for generic  $x \in \mathbb{C}^\times$ where 
\begin{align}
\begin{split}
&\bullet \textrm{$M$ is the knot exterior $\mathcal{M} \backslash \nu(\mathcal{K})$;}
\\
&\bullet\textrm{$\gamma \in H_1(\partial M;\mathbb{Z})$ is a cycle representing a meridian of $\mathcal{K}$;}
\\
&\bullet\textrm{$\mathrm{tr}_\gamma : X^\mathrm{irr}(M) \rightarrow \mathbb{C},\  [\rho] \mapsto \mathrm{tr}(\rho(\gamma))$ is the trace function of $\gamma$;}
\\
&\bullet\textrm{$d_\gamma = \begin{cases} 
1 \quad \textrm{if $\gamma \in \textrm{Ker}(i_\ast : H_1 (\partial M;\mathbb{Z}) \rightarrow H_1 (M;\mathbb{Z}/2\mathbb{Z})) $} \\ 2 \quad \textrm{otherwise}.
    \end{cases}$}
\end{split}
\label{eqn:notation}
\end{align}
Here $\mathrm{Tor}(M;\mathfrak{g}_\rho,\gamma)$ is the adjoint Reidemeister torsion with respect to $\rho$ and $\gamma$. We refer to Section \ref{sec:def} for the precise definition.
An interesting point is that there is no other topological quantity  appeared in the relation \eqref{3d-3d relation} except the adjoint Reidemeister torsion, which is nothing but the 1-loop perturbative invariant of $\mathrm{SL}_2(\mathbb{C})$ Chern-Simons theory.

\section{Ground state counting of  wrapped M5-branes} \label{sec: conjecture from physics}
 
% \subsection{SUSY ground state counting} 

 The 3D twisted index \eqref{twistedindex} can be also viewed as the Witten index \cite{Witten:1982df} for the 1D supersymmetric quantum mechanics (SQM)
\begin{align*}
\begin{split}
&\mathcal{T}_N[\mathcal{M}\times \Sigma_g,\mathcal{K}]:= (\textrm{1D SQM obtained from a twisted compactification of}
\\
&\qquad \textrm{the 6D $A_{N-1}$ (2,0) theory along $\mathcal{M}\times \Sigma_g$ with a regular defect along $\mathcal{K}\times \Sigma_g$}). \end{split}
\end{align*}
We hereafter restrict our attention to $N=2$ and omit the subscript $N$ for simplicity.
The symmetries of $\mathcal{T}[\mathcal{M}\times \Sigma_g;\mathcal{K}]$ are 
\begin{align} 
\label{symmetries}
\begin{split} 
&\bullet \textrm{Time translation},
\\
&\bullet\textrm{2  supercharges $Q$s :  $16$  $Q$s are broken to 2 $Q$s by  topological twisting},
\\
&\bullet\textrm{$SO(2)$ R-symmetry :  $SO(5)$  is broken to $SO(2)$ by topological twisting},
\\
&\bullet\textrm{$su(2)$  symmetry :  flavor symmetry associated to the knot $\mathcal{K} \subset \mathcal{M}$}. 
\end{split}
\end{align}
Here the symbol $su(2)$ could be either $SO(3)$ or $SU(2)$. The global structure of the $su(2)$ symmetry of $\mathcal{T}[\mathcal{M}\times \Sigma_g;\mathcal{K}]$ is determined by the following criterion \cite{Gang:2018wek}.
\begin{align*}
su(2)  =  \begin{cases} SU(2) \quad \textrm{if $\gamma \in \textrm{Ker}\: i_* $} \\ SO(3) \quad \textrm{otherwise.} 
\end{cases}
\end{align*}
% Recall the definition \eqref{eqn:notation} that $\gamma$ is a meridian cycle of $\mathcal{K}$.

%
The Witten index \cite{Witten:1982df} of $\mathcal{T}[\mathcal{M}\times \Sigma_g;\mathcal{K}]$ is given by
\begin{align}\label{eqn:witten}
\mathcal{I}(x;\mathcal{M}, \mathcal{K}, g):= \textrm{Tr}_{\mathcal{H} [\mathcal{M}\times \Sigma_g;\mathcal{K}]}\; e^{- \beta \hat{H}}(-1)^{\hat{R}}  x^{\hat{T}}.
\end{align}
Here the trace is taken over  an infinite dimensional vector space
\begin{align*}
\mathcal{H} [\mathcal{M}\times \Sigma_g;\mathcal{K}]\;: \; \textrm{Hilbert-space of $\mathcal{T}[\mathcal{M}\times \Sigma_g;\mathcal{K}]$}
\end{align*}
and the Noether charge operators $\hat{H}, \hat{R}$, and $\hat{T}$ associated to the symmetries in \eqref{symmetries} act on $\mathcal{H} [\mathcal{M}\times \Sigma_g;\mathcal{K}]$ as mutually commuting self-adjoint operators.
\begin{align*}
\begin{split}
&\bullet \hat{H} \;: \; \textrm{Energy}
\\
&\bullet \hat{R}\; :\; \textrm{$SO(2)=U(1)$ R-symmetry charge}
\\
&\bullet \hat{T} \;:\;  \textrm{A cartan of  $su(2)$ flavor symmetry}
\end{split}
\end{align*}
We choose a normalization of Cartan generators $R$ and $T$ of $U(1)$ and $su(2)$ respectively as follows.
\begin{align*}
\begin{split}
&i R = i \; {\rm Id}_{1\times 1} \in {\mathfrak u}(1)
\\
& i T = i  \begin{pmatrix} 
1 & 0 \\
0 & -1 
\end{pmatrix}   \in \mathfrak{su}(2)
\end{split}
\end{align*}
Note that the operators $\hat{R}$ and $\hat{T}$ act linearly and their eigenvalues are
\begin{align}
\begin{split}
&(\textrm{Eigenvalues of $\hat{H}$}) \in \mathbb{R}_{\geq 0}, \quad (\textrm{Eigenvalues of $\hat{R}$}) \in \mathbb{Z},\\
& (\textrm{Eigenvalues of $\hat{T}$}) \in   \begin{cases} \mathbb{Z} & \textrm{if $\gamma \in \textrm{Ker}\: i_* $}\quad \\ 2 \mathbb{Z} & \textrm{otherwise.} 
\end{cases}
\end{split} \label{Spectrum of Noether charges}
\end{align}
 On the Hilbert-space $\mathcal{H}[\mathcal{M}\times \Sigma_g;\mathcal{K}]$, there are also Grassmannian odd supercharge operators, $\hat{Q}$ and its adjoint $\hat{Q}^\dagger$, satisfying the (anti)-commutation relations:
\begin{align*}
\begin{split}
&\hat{Q} \hat{Q}^\dagger + \hat{Q}^\dagger \hat{Q} = 2\hat{H}\;, \quad [\hat{Q},\hat{T}] = [\hat{Q}^\dagger, \hat{T}] = \hat{Q}^2  = 0\;,
\\
& [\hat{R},\hat{Q}] = \hat{Q}\;, \quad [\hat{R}, \hat{Q}^\dagger] = - \hat{Q}^\dagger\;.
\end{split}
\end{align*}

As the usual, the index \eqref{eqn:witten} does not depend on $\beta$, since there is no contribution other than supersymmetric ground states $(\hat{H}=0)$. 
More precisely, this is due to the cancellation between two states
\begin{align*}
|E, R,T \rangle  \quad \textrm{and } \quad \sqrt{\frac{2}{E}}\hat{Q}|E,R,T\rangle
\end{align*}
for $E\neq 0$ where $|E,R,T\rangle$ is a normalized simultaneous eigenstate of $(\hat{H},\hat{R},\hat{T})$ with eigenvalues $(E,R,T)$ respectively.  The second state is also  a normalized simultaneous eigenstate  with eigenvalues $(E,R+1,T)$.  
Therefore, the index \eqref{eqn:witten} counts the ground states of $\mathcal{T}[\mathcal{M}\times \Sigma_g;\mathcal{K}]$ with signs, i.e.
\begin{align}
    \mathcal{I}(x;\mathcal{M}, \mathcal{K} , g):= \textrm{Tr}_{\mathcal{H}^{E=0} [\mathcal{M}\times \Sigma_g;\mathcal{K}]}\; (-1)^{\hat{R}}  x^{\hat{T}}
    \label{eqn:ground}
\end{align}
where
\begin{align*}
 \mathcal{H}^{E=0} [\mathcal{M}\times \Sigma_g;\mathcal{K}]  := \left\{|\psi \rangle \in \mathcal{H} [\mathcal{M}\times \Sigma_g;\mathcal{K}] \;: \; \hat{H}|\psi \rangle =0     \right\}.
\end{align*}
Note that the condition $\hat{H}|\psi \rangle =0$ is equivalent to $\hat{Q}|\psi \rangle=0$.
Under the assumption that every irreducible component of $X^\mathrm{irr}(M)$ is of dimension 1, we expect that the number of  ground states  (= \textrm{dim} $\mathcal{H}^{E=0} [\mathcal{M}\times \Sigma_g;\mathcal{K}]$) is finite. It follows that (using the fact \eqref{Spectrum of Noether charges})
\begin{align}
\mathcal{I}(x;\mathcal{M}, \mathcal{K},g)   \in   \begin{cases} \mathbb{Z}[x+x^{-1}] &\textrm{if $\gamma \in \textrm{Ker}\;i_* $}  \\ \mathbb{Z}[x^2+x^{-2}] & \textrm{otherwise}.
\end{cases} \label{3d Index as Laurent series}
\end{align}
 Combining the above with the 3D-3D relation \eqref{3d-3d relation}, we obtain a non-trivial 
prediction: (for simplicity we substitute $x+x^{-1}$ by $z$)
  \begin{conjecture} Let $M$ be a compact 3-manifold with a torus boundary.	Suppose that every irreducible component of $X^\mathrm{irr}(M)$ is of dimension 1. Then for any slope $\gamma \in H_1(\partial M; \mathbb{Z})$ we have
 	\[
    \sum_{[\rho] \in \mathrm{tr}_\gamma^{-1}(z)} \mkern-10mu\left( d_\gamma  \cdot \textrm{Tor}(M;\mathfrak{g}_\rho, \gamma) \right)^{g-1}   \in   \begin{cases} \mathbb{Z}[z] & \textrm{if $\gamma \in\textrm{Ker}\;i_* $} \\ \mathbb{Z}[z^2] & \textrm{otherwise}
 	\end{cases}\]
 	for generic $z \in \mathbb{C}$. See the notation \eqref{eqn:notation}.
 \end{conjecture}

	\begin{remark}
			The geometrical choice of $(\mathcal{M},\mathcal{K})$ determines and is determined by the other pair $(M, \gamma)$ via Dehn filling and drilling out. 
			In Sections \ref{sec:def} and \ref{sec:proof},
			we use the latter pair in order to follow mathematical conventions.
	\end{remark}
\subsection{Conjecture for $g=0$ : no SUSY  ground state } \label{sec : no ground state}
We  claim that if the knot complement $\mathcal{M}\backslash\mathcal{K}$ (or equivalently the interior of $M$) admits a hyperbolic structure,  then
\begin{align} \label{eqn:empty}
\begin{split}
&\mathcal{H}^{E=0}[\mathcal{M}\times \Sigma_g;\mathcal{K}] =(\textrm{empty}) \quad \textrm{for } g=0.
% & \textrm{for all hyperbolic $M$ and   slope $\gamma \in H_1 (\partial M, \mathbb{Z})$}
\end{split}
\end{align}
As a consequence, we obtain Conjecture \ref{conj:intro} from the relations \eqref{3d-3d relation} and \eqref{eqn:ground}.

A physical argument for the claim \eqref{eqn:empty} is given as follows. We may consider general $N \geq 2$.
% Now let us give physical argument  why we expect $\textrm{dim}\mathcal{H}_N^{E=0}[M\times \Sigma_g;\gamma]_{g=0}=0$ for general $N\geq 2$.
The Hilbert-space describes  supersymmetric ground states of the 1D quantum mechanical system obtained from a twisted compactification of 6D $A_{N-1}$ theory on $\mathcal{M}\times \Sigma_g$ with a maximal regular defect along $\mathcal{K} \subset \mathcal{M}$.  The 3D-3D relation was derived by studying the twisted index on $\Sigma_g$ of the 3D theory $\mathcal{T}_{N}[\mathcal{M},\mathcal{K}]$. Alternatively, the ptn can be considered as the twisted index on $M$ of the 4D theory $\mathcal{T}_N[\Sigma_g]$  \cite{Gaiotto:2008cd}. The 4D theory is defined as the low energy effective theory of the twisted compactification of 6D $A_{N-1}$ theory on $\Sigma_g$.   The classical vacuum moduli space of the 4D theory is given as the solutions of generalized $\mathrm{SL}_N(\mathbb{C})$ Hitchin's equations coupled to a real adjoint scalar $\sigma$  on $\Sigma_g$ \cite{Yonekura:2013mya}. When $g=0$, there is only the trivial flat connection on $\Sigma_g = S^2$ and the Hitchin moduli space is a point. On the point, there is $(N-1)$-dimensional vacuum moduli (say reducible branch) space parameterized by the real scalar field $\sigma$.   Physically, the moduli describes the dynamics of totally separated $N$ M5-branes. As argued in \cite{Gang:2018wek}, the reducible branch  is separated from the other branch (say irreducible branch) of vacuum moduli space  in the compactification of the  6D theory on the hyperbolic 3-manifold $M$. The 3D-3D relation \eqref{3d-3d relation} was derived for the 3D theory sitting on the  irreducible branch. After siting on the irreducible branch, the reducible vacua disappear and there is no remaining classical vacua for $g=0$. This implies the claim \eqref{eqn:empty}.

\section{The Reidemeister torsion for a knot exterior} \label{sec:def}

\def\Zbb{\mathbb{Z}}
\def\Fbb{\mathbb{F}}
\def\Rbb{\mathbb{R}}
\def\Cbb{\mathbb{C}}

We devote this section to briefly recall basic definitions and known results for the sign-refined Reidemeister torsion. We mainly follow \cite{porti1997torsion, yamaguchi2008relationship,turaev2012torsions}.

\subsection{Definitions}

\subsubsection{Torsion of a chain complex} \label{sec:chaincplx}

Let $\Fbb$ be a field.
For an $\Fbb$-vector space with two (ordered) bases $c$ and $c'$ we denote by $[c'/c] \in \Fbb^\times$ the determinant of the transition matrix taking $c$ to $c'$.

Let $(0 \rightarrow C_n \rightarrow \cdots \rightarrow C_0\rightarrow 0)$ be a chain complex of $\Fbb$-vector spaces with boundary maps $\partial_i : C_i \rightarrow C_{i-1}$.
Let $c_i$ be a basis of $C_i$ and $h_i$ be a basis of the $i$-th homology $H_i(C_\ast)$.
We choose $b_i \subset C_i$ such that $\partial_i b_i$ is a basis of $\mathrm{Im} \mkern1mu \partial_{i} \subset C_{i-1}$ and choose a representative $\widetilde{h}_i$ of $h_i$ in $\mathrm{Ker}\mkern1mu \partial_i \subset C_i$. 
It follows from the short exact sequences
\begin{equation*}
		0\rightarrow \mathrm{Ker}\mkern1mu \partial_i \rightarrow C_i  \overset{\partial_i}{\rightarrow} \mathrm{Im}\mkern1mu \partial_i \rightarrow 0 \quad \textrm{and}\quad 
		0 \rightarrow \mathrm{Im}\mkern1mu \partial_{i+1} \rightarrow \mathrm{Ker}\mkern1mu \partial_i\rightarrow H_i(C_\ast)\rightarrow 0
\end{equation*}
that the collection $\partial_{i+1} b_{i+1} \sqcup \widetilde{h}_i \sqcup b_i$ is a basis of $C_i$.
The \emph{sign-refined Reidemeister torsion} (\cite{turaev1986reidemeister}) is defined by
\begin{equation}\label{eqn:torsion}
    \mathrm{Tor}(C_\ast,c_\ast, h_\ast) =  (-1)^{|C_\ast|}\prod_{i=0}^n \left[ (\partial_{i+1} b_{i+1} \sqcup \widetilde{h_i} \sqcup b_i)/c_i \right]^{(-1)^i} \in \Fbb^\times
\end{equation}
where the symbol $|C_\ast|= \sum_{i=0}^n \big(\sum_{j=0}^i \dim C_j\big) \cdot \big(\sum_{j=0}^i \dim H_j(C_\ast)\big) $. 
It is known (see e.g. \cite{turaev1986reidemeister,turaev2012torsions}) that $\mathrm{Tor}(C_\ast,c_\ast,h_\ast)$ does not depend on the auxiliaries choices of $b_\ast$ and $\widetilde{h}_\ast$.

\subsubsection{Torsion of a CW-complex} \label{sec:cw}

Let $W$ be a finite CW-complex with a \emph{homology orientation} $\mathfrak{o}$,	an orientation of the $\Rbb$-vector space $H_\ast(W;\Rbb)$.
We enumerate the cells of $W$ by $c_i$ ($1 \leq i \leq m)$
and fix an orientation of each $c_i$ so that 
the cells $c_\ast$ form a basis of $C_\ast(W;\Rbb)$. 
We choose a basis $h_\ast$ of $H_\ast(W;\Rbb)$ respecting the orientation $\mathfrak{o}$ and let 
\begin{equation}\label{eqn:tau}
    \tau=\mathrm{sgn} \left( \mathrm{Tor}(C_\ast(W;\Rbb),c_\ast,h_\ast)\right) \in \{\pm1\}.
\end{equation}
The sign $\tau$ clearly depends on the orientation $\mathfrak{o}$, but not on a precise choice of the basis $h_\ast$.

Let $G=\mathrm{SL}_2(\Cbb)$ and $\mathfrak{g}$ be its Lie algebra.
A representation $\rho:\pi_1(W)\rightarrow G$ endows $\mathfrak{g}$ with 
a right $\Zbb[\pi_1 (W)]$-module structure: $v \cdot g = \mathrm{Ad}_{\rho(g^{-1})} (v)$  for $v \in \mathfrak{g}$ and $g \in \pi_1(W)$. 
We consider the chain complex
\[ C_*(W;\mathfrak{g}_\rho) = \mathfrak{g} \otimes_{\Zbb[\pi_1 W]} C_*(\widetilde{W};\Zbb)\]
where $\widetilde{W}$ is the universal cover of $W$ with the induced CW-structure, and denote by $H_\ast(W;\mathfrak{g}_\rho)$ its homology.

For each cell $c_i$ of $W$ we choose a lift $\widetilde{c}_i$ to $\widetilde{W}$ arbitrarily so that the set
\begin{equation*}
\mathcal{B} =\{ h \otimes \widetilde{c}_1, e \otimes \widetilde{c}_1,  f \otimes \widetilde{c}_1,  \cdots, \ h \otimes \widetilde{c}_m, e \otimes \widetilde{c}_m,  f \otimes \widetilde{c}_m \}
\end{equation*}
is a basis of $C_\ast (W;\mathfrak{g}_\rho)$. 
Here  $\{h,e,f\}$ is the usual basis of $\mathfrak{g}$ (see Remark \ref{rmk:basis_ind}). Letting $\boldsymbol{h}_\ast$ be a basis of $H_\ast(W;\mathfrak{g}_\rho)$, we define
 \begin{equation*} 
    \mathrm{Tor}(W;\mathfrak{g}_\rho, \boldsymbol{h}_\ast, \mathfrak{o}) :=\tau \cdot \mathrm{Tor}(C_\ast(W;\mathfrak{g}_\rho), \mathcal{B}, \boldsymbol{h}_\ast) \in \Cbb^\times.
\end{equation*}
 The torsion $\mathrm{Tor}(W;\mathfrak{g}_\rho, \boldsymbol{h}_\ast, \mathfrak{o})$ does not depend on the choice of the order and orientations of the cells $c_\ast$ (as they appear both in $\tau$ and $\mathcal{B}$) and the lifts $\widetilde{c}_i$ (as we are working on $\mathrm{SL}$). 
 It is also known that $\mathrm{Tor}(W;\mathfrak{g}_\rho,\boldsymbol{h}_\ast,\mathfrak{o})$ is invariant under conjugating $\rho$ and subdividing $W$; the sign $(-1)^{|C_\ast|}$ in the equation \eqref{eqn:torsion} is needed to ensure such invariance. 
 We refer to \cite{turaev2001introduction} for details.
 
\begin{remark} \label{rmk:basis_ind} It is known that if the Euler characteristic of $W$ is zero, the torsion
$\mathrm{Tor}(W;\mathfrak{g}_\rho,\boldsymbol{h}_\ast,\mathfrak{o})$  does not depend on the choice of the basis of $\mathfrak{g}$. We hereafter only consider knot exterior in $S^3$, so the basis of $\mathfrak{g}$ would not be essential.
\end{remark}

 \subsubsection{Torsion of a knot exterior} \label{sec:knotexterior}
 Let $K$ be an oriented hyperbolic knot in $S^3$ and $M$ be the knot exterior $S^3 \backslash \nu(K)$ with a fixed triangulation. We orient a meridian $\mu$ of $K$ by the right-hand screw rule and let a homology orientation $\mathfrak{o}$ of $M$ be the one induced from the basis
 $\{[\mathrm{pt}], [\mu]\} \subset H_\ast(M;\Rbb)=H_0(M;\Rbb) \oplus H_1(M;\Rbb)$. Here $\mathrm{pt}$ denotes a point in $M$.

A \emph{slope} $\gamma$ is an oriented simple closed curve in $\partial M$ with non-trivial class in $H_1(\partial M;\Zbb)$. An irreducible representation $\rho :\pi_1(M)\rightarrow G$ is called \emph{$\gamma$-regular} (\cite{porti1997torsion,yamaguchi2008relationship}) if:
 \begin{itemize}
     \item $\dim  H_1(M;\mathfrak{g}_\rho)=1$; 
     \item the inclusion $\gamma \hookrightarrow M$  induces an epimorphism $H_1(\gamma;\mathfrak{g}_\rho) \rightarrow H_1(M;\mathfrak{g}_\rho)$;
     \item if $\mathrm{tr} \left(\rho(\pi_1(\partial M)) \right)  \subset \{ \pm 2\}$, then $\rho(\gamma) \neq \pm\mathrm{Id}$.
\end{itemize} 
The orientation of $\gamma$ is not necessary here but is required in the construction below. Note that the notion of $\gamma$-regularity is invariant under conjugation, so the notion of an irreducible $\gamma$-regular character is well-defined.

For an irreducible $\gamma$-regular representation $\rho$ it is known that $\dim H_i(M;\mathfrak{g}_\rho)=1$ for $i= 1,2$, and $\dim H_i(M;\mathfrak{g}_\rho)=0$ for $i\neq 1,2$.
We fix a basis $\boldsymbol{h}_\ast=\{\boldsymbol{h}_1,\boldsymbol{h}_2\}$ of $H_\ast(M;\mathfrak{g}_\rho)$
by $\boldsymbol{h}_1 = v \otimes \widetilde{\gamma}$ and $\boldsymbol{h}_2 = v \otimes \widetilde{\partial M}$
where $v\in\mathfrak{g}$ is a non-zero vector  such that $\mathrm{Ad}_{\rho(g)}(v)=v$ for all $g \in \pi_1(\partial M)$.
Note that (i) such a vector $v$ is unique up to scaling, since $\rho(\pi_1(\partial M)) \not \subset \{\pm \mathrm{Id}\}$;
(ii) the tilde symbols $\widetilde{\gamma}$ and $\widetilde{\partial M}$ are used for lifts of $\gamma$ and $\partial M$ to the universal cover;
(iii) the boundary $\partial M$ is oriented by the convention ``the inward  normal vector in the last position''.
With above choices, we define
\[ \mathrm{Tor}(M;\mathfrak{g}_\rho,\gamma) := \mathrm{Tor}\left(M;\mathfrak{g}_\rho,\boldsymbol{h}_\ast, \mathfrak{o}\right) \in \Cbb^\times.\]
The torsion $\mathrm{Tor}(M;\mathfrak{g}_\rho,\gamma)$ does not depend on the choice of a triangulation of $M$ (as the torsion is invariant under subdivision) and a scaling of the vector $v$ (as it appears in both $\boldsymbol{h}_1$ and $\boldsymbol{h}_2$).
We refer to \cite[\S 3]{porti1997torsion} and \cite{yamaguchi2008relationship} for details.

\subsubsection{Changing a slope} \label{sec:changing}
Suppose that an irreducible representation  $\rho :\pi_1(M)\rightarrow G$  is both $\gamma$- and $\delta$-regular for slopes $\gamma$ and $\delta$.
It is known (see \cite[\S 3]{porti1997torsion}) that its character $[\rho]$ is contained in an 1-dimensional irreducible component, say $X$, of the algebraic set $X^\mathrm{irr}(M)=\{\pi_1(M)\rightarrow G : \mathrm{irreducible}\}/_{\mathrm{Conj}}$.
Let $u_\gamma$ and $u_\delta : X \rightarrow \Cbb$ be functions satisfying up to conjugation
\begin{equation} \label{eqn:fcnu}
\rho'(\gamma) = \begin{pmatrix} e^{u_\gamma([\rho'])} & \ast \\ 0 & e^{-u_\gamma([\rho'])} \end{pmatrix} 
\quad \mathrm{and} \quad
\rho'(\delta) = \begin{pmatrix} e^{u_\delta([\rho'])} & \ast \\ 0 & e^{-u_\delta([\rho'])} \end{pmatrix} 
\end{equation}
 for all characters $[\rho']\in X$. 
Under the assumption that both $u_\gamma$ and $u_\delta$ are holomorphic and non-singular at $[\rho]$, which is often the case,  ``slope changing rule'' follows  from \cite[Proposition 4.7]{porti1997torsion}:
\begin{equation}\label{eqn:basischange}
\mathrm{Tor}(M;\mathfrak{g}_\rho,\gamma) = \dfrac{\partial  u_\gamma}{\partial  u_\delta} \cdot \mathrm{Tor}(M;\mathfrak{g}_\rho,\delta)
\end{equation}
where the derivative is evaluated at $[\rho]$.

\subsection{Formulas for computing the torsion}

The torsion of a knot exterior is often computed by using the torsion polynomial (see e.g. \cite{yamaguchi2008relationship,dubois2009non,tran2014twisted}). It is computationally convenient as the torsion polynomial is obtained from a chain complex with trivial homology.

\subsubsection{Torsion polynomial}

Let $K\subset S^3$ be an oriented hyperbolic knot and $M=S^3 \backslash \nu(K)$.
Let $\rho:\pi_1(M)\rightarrow G$ be an irreducible representation and $\alpha : \pi_1(M)\rightarrow \Zbb$ be the abelianization map, counting the signed linking number with $K$.
We endow $\mathfrak{g}(t)=\Cbb(t)\otimes \mathfrak{g}$ with a right $\Zbb[\pi_1(M)]$-module structure: 
\[(p \otimes v)\cdot g =  t^{\alpha(g)}p\otimes \mathrm{Ad}_{\rho(g^{-1})}(v)\]for $p \otimes v \in \mathfrak{g}(t)$ and $g \in \pi_1(M)$. Here $\Cbb(t)$ denotes the field of rational functions in one variable $t$.

Let $\lambda$ be the canonical longitude of $K$ with the orientation same as $K$, and assume that $\rho$ is $\lambda$-regular. It is known (see \cite[Proposition 3.1.1]{yamaguchi2008relationship}) that the chain complex   
 \[C_\ast(M;\mathfrak{g}(t)_\rho)=\mathfrak{g}(t) \otimes_{\Zbb[\pi_1 M]} C_\ast(\widetilde{M};\Zbb)\]
 of $\Cbb(t)$-vector spaces is acyclic, i.e., its holomogy is trivial.
The \emph{torsion polynomial} is given by
 \[\mathrm{Tor}(M;\mathfrak{g}(t)_\rho) := \tau \cdot \mathrm{Tor}(C_\ast(M;\mathfrak{g}(t)_\rho), 1 \otimes \mathcal{B}, \emptyset) \in \Cbb(t)^\times.\]
We refer to Section \ref{sec:cw} for definitions of the sign $\tau$ and basis $\mathcal{B}$. 
Note that the torsion polynomial is defined up to $t^n (n\in\Zbb)$ due to the $\pi_1(M)$-action on $\Cbb(t)$.

In \cite{yamaguchi2008relationship} Yamaguchi proved that the torsion polynomial determines $\mathrm{Tor}(M;\mathfrak{g}_\rho,\lambda)$ as follows.
\begin{theorem}[\cite{yamaguchi2008relationship}] \label{thm:yamaguchi} The torsion polynomial $\mathrm{Tor}(M;\mathfrak{g}(t)_\rho)$ has  a simple zero at $t=1$ and
\begin{equation} \label{eqn:yamaguchi}  \mathrm{Tor}(M;\mathfrak{g}_\rho,\lambda)=-\left. \dfrac{d}{dt}\right|_{t=1} \mathrm{Tor}(M;\mathfrak{g}(t)_\rho).
\end{equation}
Note that the indeterminacy of $\mathrm{Tor}(M;\mathfrak{g}(t)_\rho)$ does not affect the equation \eqref{eqn:yamaguchi}.
\end{theorem}

\subsubsection{Fox calculus} In \cite{kitano1996twisted} Kitano proved that the torsion polynomial agrees with the twisted Alexander invariant of $K$ with respect to $\mathrm{Ad}_\rho$. In particular, it can be computed in terms of the Fox differential calculus.

We first choose a finite group presentation of $\pi_1(M)$ of  deficiency 1  (for instance, the Wirtinger presentation)
\[\pi_1(M) = \langle g_1,\cdots,g_n | r_1,\cdots,r_{n-1} \rangle.\]
Let $W$ be the 2-dimensional CW-complex corresponding to the presentation.
Recall that $W$ has one 0-cell, $n$ 1-cells, and $(n-1)$ 2-cells.
Identifying the chain complex $C_\ast(W;\mathfrak{g}(t)_\rho)$ with
\[
    0 \rightarrow \mathfrak{g}(t)^{n-1}   \overset{\partial_2}{\longrightarrow} \mathfrak{g}(t)^{n} \overset{\partial_1}{\longrightarrow} \mathfrak{g}(t) \rightarrow 0,
\]
the boundary maps $\partial_1$ and $\partial_2$ are given by (see e.g. \cite{kitano1996twisted}, \cite[\S 3.4]{yamaguchi2008relationship})
\begin{align*}
\partial_1 &= \left( \Phi(g_1 -1),  \cdots ,\Phi(g_n-1)\right) \in M_{1,n}(M_{3,3}(\Cbb(t)))     \\
\partial_2 &=  \left( \Phi\left(\dfrac{ \partial r_i}{\partial g_j}\right)\right)_{1 \leq i \leq n-1,\   1\leq j \leq n} \in M_{n,n-1}(M_{3,3}(\Cbb(t))).
\end{align*}
Here $\partial/\partial g$ denotes the Fox calculus, $M_{i,j}$ denotes the set of $i\times j$ matrices, and  the map $\Phi : \Zbb[\pi_1(M)] \rightarrow M_{3,3}(\Cbb(t))$ is given by \[\Phi \left(\sum_{i} n_i g_i\right)=\sum_{i} n_i t^{\alpha(g_i)} \mathrm{Ad}_{\rho(g_i)} \in M_{3,3}(\Cbb(t))\]
for $n_i \in\Zbb$ and $g_i \in \pi_1(M)$.

From the fact that $W$ is simple homotopic to the knot exterior $M$, we have 
\begin{equation} \label{eqn:foxtor}
\mathrm{Tor}(M;\mathfrak{g}(t)_\rho) = \epsilon \cdot  \dfrac{\det \left( \partial_{2;\widehat{j}} \right)}{\det \left(\Phi(g_{j}-1)\right)}
\end{equation}
for some $1 \leq j \leq n$ and $\epsilon \in \{ \pm1\}$. Here $\partial_{2;\widehat{j}}$ is the square matrix obtained from $\partial_2$ by deleting the $j$-th row, and the index $j$ can be chosen freely among those satisfying  $\det (\Phi(g_j-1)) \neq 0$. The existence of such an index $j$ follows from  \cite[Lemma 2]{wada1994twisted}.
\begin{remark} \label{rmk:sign_ind}
The sign $\epsilon$ in the equation \eqref{eqn:foxtor} only depends on the choice of $j$ and whether the simple homotopy $W\rightarrow M$ preserves the homology orientation. See \cite[Remark 2.4]{turaev2001introduction} and \cite[Theorem 18.3]{turaev2001introduction}. 
\end{remark}

%%%%%%%%%%%%%%%%%%%%%%%%%%%%%%%%%%%%%%%%%%%%%%%%%%%%%%%%%%%%%

\section{Supporting evidences for Conjecture \ref{conj:intro}} \label{sec:proof}

In this section, we present a rigorous proof of Conjecture \ref{conj:intro} for the $4_1$ knot:
\begin{theorem} \label{thm:41} Let $M$ be the knot exterior of the $4_1$ knot. Let $\gamma\in H^1(\partial M; \Zbb)$ be any slope and $\mathrm{tr}_\gamma : X^{\mathrm{irr}}(M)\rightarrow \Cbb$ be its trace function.
Then we have
\[ \sum_{[\rho] \in \mathrm{tr}_\gamma^{-1}(z)} \dfrac{1}{\mathrm{Tor}(M;\mathfrak{g}_\rho,\gamma)}=0\]
for generic $z \in \Cbb$. 
\end{theorem}
The proof relies on some facts related to global residue theorem. They are briefly reviewed in Section \ref{sec:grt}. We refer to \cite{khovanskii1977newton, tsikh1992multidimensional} for detail .
We also present numerical verification of the conjecture for several knots in Section \ref{sec:num_comp}.

\subsection{A residue theorem for Laurent polynomials}  \label{sec:grt}

Let $f=(f_1,\cdots,f_n)$ be a system of $n$ Laurent polynomials in $n$ variables $z=(z_1,\cdots,z_n)$.
We denote by $Z_f$ the zero set $\{a \in (\Cbb^\times)^n : f_1(a) = \cdots= f_n(a)=0\}$, and say that $a \in Z_f$ is \emph{simple} if the Jacobian
\[\mathrm{Jac}_f = \det \left(\dfrac{\partial (f_1,\cdots,f_n)}{\partial (z_1,\cdots,z_n)}\right)\]
is non-zero at the point $a$.
 
With the usual notation $z^\alpha=z_1^{\alpha_1}\cdots z_n^{\alpha_n}$ for $\alpha=(\alpha_1,\cdots,\alpha_n) \in \Zbb^n$, we write $f_i=\sum_{\alpha} c_{i}^{\alpha}z^\alpha$ with the coefficients $c_i^\alpha \in \Cbb$.
The \emph{Newton polyhedron} $\Delta(f_i)$ of $f_i$ is the convex hull in $\Rbb^n$ of the set $\{ \alpha \in \Zbb^n : c_i^\alpha \neq 0\}$.
For a non-zero vector $\beta \in \Rbb^n$ we denote by $\Delta^\beta(f_i)$ the face of $\Delta(f_i)$ on which the function $\langle \cdot, \beta \rangle : \Rbb^n \rightarrow \Rbb$ attains a minimum. Here $\langle \cdot ,\cdot\rangle$ denotes the usual inner product on $\Rbb^n$. 
The system $f=(f_1,\cdots,f_n)$ is said to be \emph{non-degenerate} (\cite{khovanskii1977newton}) if  the system  \[f^\beta=(f_1^\beta, \cdots, f_n^\beta), \quad f_i^\beta=\sum_{\alpha \in \Delta^\beta(f_i)} c_i^\alpha z^\alpha,\]
has only simple zeros in $(\Cbb^\times)^n$ for any non-zero $\beta \in \Rbb^n$.

 In \cite{khovanskii1977newton} Khovanskii gave a generalization of the global residue theorem as follows (see also \cite[\S 7]{tsikh1992multidimensional}).

\begin{theorem}[\cite{khovanskii1977newton}] \label{thm:ejt} Let $f=(f_1,\cdots,f_n)$ be a non-degenerate system of Laurent polynomials. Suppose that all of the zeros $ a =(a_1,\cdots, a_n)\in Z_f$ are simple. Then for any Laurent polynomial $h$ whose Newton polyhedron $\Delta(h)$ lies strictly  inside the (Minkowski) sum $\Delta(f_1) +\cdots+\Delta(f_n)$, we have 
\[\sum_{a\in Z_f} \dfrac{h(a)}{a_1\cdots a_n  \cdot \mathrm{Jac}_f(a)}=0 .\]
\end{theorem}

\subsection{Proof of Theorem \ref{thm:41}} \label{sec:41}

We orient $K=4_1$ as in Figure \ref{fig:diagram41} and fix a presentation of the fundamental group of the knot exterior $M$ by 
\[\pi_1(M) = \langle g_1, g_2 | g_1^{-1} g_2 g_1 g_2^{-1} g_1 - g_2 g_1^{-1} g_2 g_1 g_2^{-1} \rangle.\] 
An irreducible representation $\rho : \pi_1(M) \rightarrow G$ is given by up to conjugation
\begin{equation} \label{eqn:repn_cor}
g_1 \mapsto \begin{pmatrix} m & 1 \\ 0 & m^{-1} \end{pmatrix}, \quad  g_2 \mapsto \begin{pmatrix} m & 0 \\ y & m^{-1} \end{pmatrix}
\end{equation}
for a pair $(y,m)\in (\Cbb^\times)^2$ satisfying
\begin{equation} \label{eqn:Riley}
	f(y,m)=(y-1)(m^2+m^{-2}) + y^2-3y+3 =0.
\end{equation}
Also, zeros $(y_1,m_1)$ and $(y_2,m_2) \in (\Cbb^\times)^2$ of $f$ correspond to the same irreducible representation up to conjugation if and only if  $y_1=y_2$ and $m_1 = m_2^{\pm1}$. Thus the set $X^\mathrm{irr}(M)$ coincides with the zero set of $f$ with quotient given by $(y,m)\sim(y,m^{-1})$. We refer to \cite{riley1984nonabelian} for details.

\begin{figure}[!h]
    \centering
    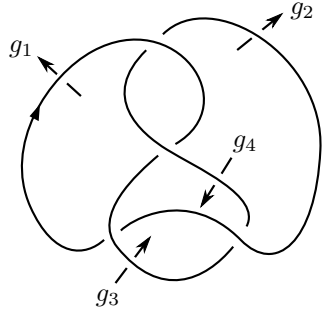
    \caption{The $4_1$ knot.}
    \label{fig:diagram41}
\end{figure}

We take a meridian $\mu$ of $K$ as the generator $g_1$. The canonical longitude $\lambda$ is given by $g_2^{-1} g_1 g_3^{-1}g_4$ from the diagram where $g_3=g_2 g_1 g_2^{-1}$ and $g_4=g_3^{-1}g_2 g_3$ (see Figure \ref{fig:diagram41}). It follows that
\[\rho(\lambda)=\rho(g_2^{-1} g_1 g_3^{-1}g_4)
=\begin{pmatrix}
 l & \ast \\ 0 & l^{-1}
\end{pmatrix}\]
where $l = - m^{-2}(y-3)(y-1)^2-m^{-4}(y^2-3y+1)$. 
Taking the resultant of $l$ with $f(y,m)=0$ to remove the variable $y$, we obtain the $\mathrm{SL}_2(\mathbb{C})$ A-Polynomial
    \[A(m,l)=l+l^{-1} +(-m^{-4}+m^{-2}+2+m^2-m^4)=0.\]
One can check that for any slope $\gamma = \mu^p \lambda^p$ ($p,q$ : coprime integers) and generic $z \in \Cbb$, the set $\mathrm{tr}_\gamma^{-1}(z)$ is identified with 
\begin{align*}
    \mathrm{tr}_\gamma^{-1}(z)=\left\{(m,l) \in (\Cbb^\times)^2: A(m,l)=0, \  B(m,l):=m^p l^q-x=0\right\}
\end{align*}
where $x \in \Cbb^\times$ is a solution to $x+x^{-1}=z$.

\begin{remark}
    The function $u_\gamma\mkern 2mu (=p \log m +q \log l)$ given as in the equation \eqref{eqn:fcnu} is non-constant on $X^\mathrm{irr}(M)$.
    It implies that 
    the set of irreducible non-$\gamma$-regular characters in $X^\mathrm{irr}(M)$ is discrete. See \cite[Proposition 3.26]{porti1997torsion} and  \cite[Remark 9]{dubois2009non}. 
    In particular, we may assume that  $\mathrm{tr}_\gamma^{-1}(z)$ consists of irreducible $\gamma$-regular characters for generic $z \in \Cbb$.  
\end{remark}

Let $r$ be the relator of the group presentation. One can compute that 
\[ \dfrac{\partial r}{\partial g_1} = -g_1^{-1}  + g_1^{-1} g_2 + g_1 ^{-1} g_2 g_1 g_2^{-1}+ g_2 g_1^{-1} - g_2 g_1^{-1} g_2.\]
It follows from the equations \eqref{eqn:foxtor} and  \eqref{eqn:yamaguchi} that
\begin{align*}
\mathrm{Tor}(M;\mathfrak{g}(t)_\rho) &= \epsilon \cdot \dfrac{\det \left(\Phi(\partial r / \partial g_1)\right)}{\det \left(\Phi(g_1-1)\right)} \\
&=\epsilon \cdot \dfrac{(t-1)(2m^2 -(t-1+t^{-1})+2m^{-2})}{t^2}
\end{align*}
and 
\begin{equation} \label{eqn:tor_lambda}
    \mathrm{Tor}(M;\mathfrak{g}_\rho,\lambda) = -\left. \dfrac{d}{dt}\right|_{t=1} \mathrm{Tor}(M;\mathfrak{g}(t)_\rho)= -\epsilon \cdot(2m^2-1+2m^{-2})
\end{equation}
for some $\epsilon \in \{ \pm1\}$.
Also, from the equation \eqref{eqn:basischange} we have
\begin{align*}
\mathrm{Tor}(M;\mathfrak{g}_\rho,\gamma) &=\dfrac{\partial(p \log m + q \log l)}{\partial \log l}\cdot\mathrm{Tor}(M;\mathfrak{g}_\rho,\lambda)\\
&= \left(p \dfrac{l}{m} \dfrac{\partial m}{\partial l} +q \right)  \cdot \mathrm{Tor}(M;\mathfrak{g}_\rho,\lambda) \\
&= \left(-p \dfrac{l}{m} \dfrac{\partial A/\partial l}{\partial A/\partial m} +q \right)  \cdot \mathrm{Tor}(M;\mathfrak{g}_\rho,\lambda) \\
&= \dfrac{l}{x} \cdot \dfrac{\det\left( \dfrac{\partial (A,B)}{\partial (m,l)}\right)}{\partial A/\partial m}  \cdot \mathrm{Tor}(M;\mathfrak{g}_\rho,\lambda)
\end{align*} 
The last equation follows from $\partial B/\partial m = p m^{p-1} l^q = px /m$ and $\partial B/ \partial l = qx/l$.
Plugging the equation \eqref{eqn:tor_lambda} and $\partial A/\partial m=-2m^{-1}(2m^2-1+2m^2)(m^2-m^{-2})$, we obtain
\begin{equation}\label{eqn:tor_gamma}
\dfrac{1}{\mathrm{Tor}(M;\mathfrak{g}_\rho,\gamma)}= 2\epsilon x  \cdot \dfrac{m^2-m^{-2}}{ml \cdot\det\left( \dfrac{\partial (A,B)}{\partial (m,l)}\right)}.
\end{equation}
\begin{remark} It is interesting that the derivative $\partial A/\partial m$ has $\mathrm{Tor}(M;\mathfrak{g}_\rho,\lambda)$ as a factor. This fact is also pointed out in \cite[Remark 4.5]{dimofte2013quantum}.
\end{remark}

We now claim that the system $(A,B)$ and the Laurent polynomial $h:=m^2-m^{-2}$ satisfy the condition of Theorem \ref{thm:ejt}.  
\begin{itemize}
	\itemsep 0.5em
    \item for any non-zero $\beta \in \Rbb^2$ the Laurent polynomial $A^\beta$ is either $l^{\pm1}$, $m^{\pm4}$, or $l^\pm - m^{\pm4}$, and the Laurent polynomial $B^\beta$ is either $x$, $m^p l^q$, or $m^p l^q-x$. Any pair of the above has only simple zeros in $(\Cbb^\times)^2$, so the system $(A,B)$ is non-degenerate.
    \item A straightforward computation shows that $\mathrm{Jac}_{(A,B)}=0$ if and only if $p(l-l^{-1})=2q(2m^2-1+2m^{-2})(m^2-m^{-2})$. Therefore, the system $(A,B)$ only has simple zeros for generic $x\in\Cbb^\times$.
    \item the Newton polygon $\Delta(A)$ strictly contains $\Delta(h)$, and $\Delta(B)$ contains the origin. Thus, $\Delta(h)$ strictly lies inside in $\Delta(A)+\Delta(B)$.
\end{itemize}
  Recall Remark \ref{rmk:sign_ind} that the sign $\epsilon$ in the equation \eqref{eqn:tor_gamma} does not depend on the choice of pair $(m,l)$. Therefore, Theorem \ref{thm:41} is obtained from Theorem \ref{thm:ejt}.
 
\subsection{Numerical verification} \label{sec:num_comp}

We here present numerical verification of Conjecture \ref{conj:intro} for $K=5_2$ and $7_4$.
\subsubsection{The $5_2$ knot}
The fundamental group of the knot exterior of $K=5_2$ is generated by two elements and 
$X^{\mathrm{irr}}(M)$ is given by the zero set of
\[f(y,m)=(y-2)(y-1)(m^2+m^{-2})+y^3-5y^2+8y-3\]
in $(\Cbb^\times)^2$ with the quotient given by $(y,m) \sim (y,m^{-1})$.
Similar computations as in Section \ref{sec:41} give that for a slope $\gamma = \mu^p \lambda^q$ and generic $x \in \Cbb$, we have
\[\mathrm{tr}_\gamma^{-1}(x) = \left\{(y,m,l) \in (\Cbb^\times)^3 : f(y,m)=0,\ g(y,m,l)=0,\ h(m,l)=0\right\}\]
where
\begin{equation*}
    \left\{ 
        \begin{array}{l}
        g(y,m,l)=l +(y-1) -m^{2}-m^{4}( y^3-5y^2+9y-4) + 2m^{6}(y-1)(y-2)  \\[7pt]
        h(m,l)=m^pl^q-x
        \end{array}.
    \right.
\end{equation*} Here $x\in \Cbb^\times$ is a solution to $x+x^{-1}= z$.
Also, the torsion $\mathrm{Tor}(M;\mathfrak{g}_\rho,\lambda)$ for the canonical longitude $\lambda$ is given by 
\[\mathrm{Tor}(M;\mathfrak{g}_\rho,\lambda) = \epsilon \cdot \dfrac{5y^3-21y^2+28y-14}{y-1}\]
for some $\epsilon \in \{\pm1\}$.
\begin{lemma} \label{lem:comp1}
    For $(y,m,l) \in \mathrm{tr}_\gamma^{-1}(x)$, we have
		\[ \dfrac{\partial (p \log m + q \log l)}{ \partial \log l}=p\dfrac{l}{m}\dfrac{\partial m}{\partial l}+q=
		 \dfrac{l}{x} \cdot \dfrac{\mathrm{det} \left( \dfrac{ \partial (f,g,h)}{\partial (y,m,l) } \right)}{\mathrm{det} \left(  \dfrac{ \partial (f,g) }{\partial (y,m)} \right)}.      \]
\end{lemma}
\begin{proof} A straightforward computation gives
\begin{align*}
    \mathrm{det} \left( \dfrac{ \partial (f,g,h)}{\partial (y,m,l) } \right) &=\mathrm{det} \left(  \dfrac{ \partial (f,g) }{\partial (y,m)} \right) \cdot \dfrac{qx}{l}  - \dfrac{\partial f}{\partial y} \dfrac{\partial g}{\partial l} \cdot \dfrac{px}{m}\\
    &=\mathrm{det} \left(  \dfrac{ \partial (f,g) }{\partial (y,m)} \right) \cdot \left ( \dfrac{qx}{l}+ \dfrac{pz}{m}  \dfrac{\partial m}{\partial l}   \right)\\
    &=\mathrm{det} \left(  \dfrac{ \partial (f,g) }{\partial (y,m)} \right) \cdot \dfrac{x}{l}  \left ( q+ \dfrac{l}{m} \dfrac{\partial m}{\partial l}  \right).
\end{align*}
The second equality follows from $df=dg=0$.
\end{proof}
One can compute that \begin{align*}
 \mathrm{det} \left(  \dfrac{ \partial (f,g) }{\partial (y,m)} \right) &= \dfrac{5y^3-21y^2+28y-14}{y-1}  \\
     & \quad \quad \cdot \dfrac{2((m^2+y-2)(y^4-6y^3+13y^2-12y+3)-(y-2)^2)}{m^9(y-1)(y-2)^3}
\end{align*}
and 
\begin{align*}
    \mathrm{Tor}(M;\mathfrak{g}_\rho,\gamma)&=\dfrac{\partial (p \log m + q \log l)}{ \partial \log l} \cdot \mathrm{Tor}(M;\mathfrak{g}_\rho,\lambda) \\
    &=\dfrac{l}{x} \cdot \dfrac{\mathrm{det} \left( \dfrac{ \partial (f,g,h)}{\partial (y,m,l) } \right)}{\mathrm{det} \left(  \dfrac{ \partial (f,g) }{\partial (y,m)} \right)} \cdot \mathrm{Tor}(M;\mathfrak{g}_\rho,\lambda)\\
    &=\dfrac{\epsilon l}{2x} \cdot \dfrac{\mathrm{det} \left( \dfrac{ \partial (f,g,h)}{\partial (y,m,l) } \right) \cdot m^9 (y-1)(y-2)^3  }{(m^2+y-2)(y^4-6y^3+13y^2-12y+3)-(y-2)^2}. \\
\end{align*}

\begin{remark} It is worth noting that the Jacobian of $(f,g)$
again contains the torsion $\mathrm{Tor}(M;\mathfrak{g}_\rho,\lambda)$ as a factor.
\end{remark}
To verify Conjecture \ref{conj:intro}, we let $F = m^2f$, $G=(y-1)(y-2)^3 m^6 g$, $H=h$, and compute the sum of
\[ \dfrac{1}{\mathrm{Tor}(M;\mathfrak{g}_\rho,\gamma)} = 2\epsilon x \cdot \dfrac{y(m^2+y-2)(y^4-6y^3+13y^2-12y+3)-y(y-2)^2}{yml \cdot \mathrm{det} \left( \dfrac{ \partial (F,G,H)}{\partial (y,m,l) } \right)} \]
over all $(y,m,l) \in \mathrm{tr}_\gamma^{-1}(z)$, which is the zero set of $(F,G,H)$. 
For instance, the set $\mathrm{tr}_\gamma^{-1}(z)$ for $(p,q)=(3,1)$ and $z=\frac{3}{2}+\frac{i}{2}$ consists of 23 points with torsions
\begin{equation*}
    \begin{array}{rrr}
        -5.1707095 + 6.056876 i, &  -5.1403791 - 5.271889 i,  & -4.9799403 + 5.257641 i, \\
        -4.7335145 - 7.299169 i, &  -4.6988457 - 5.941816 i, & -4.3082655 + 7.042614 i, \\
        -3.8808087 - 6.974908 i, &  -3.3630233 + 7.605688 i, & -2.6624296 + 3.284613 i, \\
        0.2005695 - 4.913042 i, & 9.8858003 + 2.112603 i, &  14.549795 + 0.213397 i,\\
        14.568149 + 0.187863 i, & 15.922137 - 0.358869 i, &  16.535205 - 0.634458 i, \\
        17.512936 + 0.306584 i , &  18.497289 - 1.694233 i, &  18.514426 - 0.117280 i,\\
        19.936167 + 0.800241 i, &  23.334158 - 0.639555 i, &  25.010603 + 1.138408 i, \\
        25.406178 + 0.241449 i, & 28.564506 - 0.402759 i&
    \end{array}
\end{equation*} whose inverse sum  is zero  numerically. 

\subsubsection{The $7_4$ knot} \label{sec:knot74}
It is known that for $K=7_4$ the algebraic set $X^\mathrm{irr}(M)$ has two irreducible components. Precisely, $X^\mathrm{irr}(M)$ is given by the zero set of $f_1(y,m) \cdot f_2(y,m)$ in $(\Cbb^\times)^2$ where
\begin{align*}
    f_1(y,m)&=(y-2)^2(m^2+m^{-2})+y^3-6y^2+12y-7     \\
    f_2(y,m)&=y(y-1)(y-2)(m^2+m^{-2})+y^4-5y^3+8y^2-4y+1
\end{align*}
with the quotient given by $(y,m)\sim(y,m^{-1})$.

We first consider irreducible representations coming from the zero set of $f_1$.
In this case, similar computations as in the previous section give 
\begin{align*} 
g_1(y,m,l) &=l-\frac{1}{m^{16}} \Big( m^{16} (5 - 2 y)+ m^{14} (5 - 3 y)+ m^{12} (-9 + 2 y)+2 m^{10} (-17 + 7 y) \\
& \hspace{5.5em}  + 8 m^8 (-3 + 2 y) + m^6 (86 - 26 y) + m^4 (219 - 98 y) \\
& \hspace{5.5em} + m^2 (-463 + 99 y + 707 y^2 + 1104 y^3 - 9638 y^4 + 34306 y^5 \\
&\hspace{8.2em} -  66187 y^6 + 75756 y^7 - 55526 y^8 + 27256 y^9 - 9116 y^{10} \\
&\hspace{8.2em} +2060 y^{11} - 302 y^{12} + 26 y^{13} - y^{14}) \\
& \hspace{5.5em}
       - (2 - y)^2 (-61 - 115 y - 66 y^2 + 321 y^3 - 
       1552 y^4 + 3558 y^5\\
&\hspace{8.2em}        - 3829 y^6 + 2282 y^7 - 809 y^8 + 
       171 y^9 - 20 y^{10} + y^{11})
\Big)    
\end{align*} 
whose resultant with $f_1$ gives one component $A_1$ of the A-polynomial
\begin{align*}
A_1(m,l)=& m^{14} + 
l (1 - 2 m^2 + 3 m^4 + 2 m^6 - 7 m^8 + 2 m^{10} + 6 m^{12} - 2 m^{14}) \\&+ 
l^2 (-2 + 6 m^2 + 2 m^4 - 7 m^6 + 2 m^8 + 3 m^{10} - 2 m^{12} + m^{14}) + l^3.
\end{align*}
Also, the torsion $\mathrm{Tor}(M,\mathfrak{g}_\rho,\lambda)$ for the canonical longitude $\lambda$ is given by
\begin{align*}
&\frac{1}{m^{24} \left(m^2-1\right)^2}
\Big(m^{28} (12-8 y)+4 m^{26} (y-2)+m^{24} (24 y-29)+m^{22} (34-30 y)\\
	&+m^{20} (20 y-31)+m^{18} (23-4 y)+m^{16} (91-43 y)+m^{14} (6-16 y)+m^{12} (99 y-287)\\
	&+m^{10} (250 y-502)+m^8 (455-61 y)+m^6 (2665-1076 y)+m^4 (2105-1354 y)\\
	&+m^2 (-7 y^{17}+212 y^{16}-2931 y^{15}+24429 y^{14}-136467 y^{13}+536999 y^{12}-1521221 y^{11}\\
	&\hspace{3em}+3111157 y^{10}-4527550 y^9+4521505 y^8-2878308 y^7+978784 y^6-53392 y^5\\
	&\hspace{3em}-77400 y^4+16421 y^3+6770 y^2+12018 y-9213)\\
	&-(y-2)^2 (7 y^{14}-170 y^{13}+1834 y^{12}-11520 y^{11}+46330 y^{10}-123292 y^9+215917 y^8\\
	&\hspace{4em}-237784 y^7+147025 y^6-37064 y^5-3126 y^4+4792 y^3-155 y^2-1322 y-1591)\Big)
\end{align*}
Similarly, from the other component $f_2$, we have
\begin{align*}
g_2(y,m,l)=l&-
\frac{1}{m^{16}}\Big( m^{24} (-(y-2)) y+m^{22} \left(y^2-3 y+1\right)-m^{18} (y-2) y+m^{16} \left(2 y^2-6 y+3\right)\\
&+m^{14} \left(3 y^2-7 y+2\right)+3 m^{12} (y-1)^2-3 m^{10} (y-3)+m^8 \left(y^2-16 y+27\right)\\
&+m^6 \left(21 y^2-77 y+64\right)+m^4 \left(71 y^2-206 y+129\right)\\
&-m^2 (y^{13}-23 y^{12}+232 y^{11}-1350 y^{10}+5022 y^9-12552 y^8+21757 y^7-27137 y^6\\
&\hspace{3em}+25878 y^5-20076 y^4+12429 y^3-5513 y^2+1582 y-260)\\
&-y (y^{11}-21 y^{10}+191 y^9-987 y^8+3201 y^7-6828 y^6+9895 y^5-10224 y^4\\
&\hspace{3em}+8164 y^3-5252 y^2+2380 y-520)\Big)
\end{align*} with the other component $A_2$ of the $A$-polynomial
\[A_2(m,l)=m^8 + l (-1 + m^2 + 2 m^4 + m^6 - m^8)+l^2 \]
and the torsion $\mathrm{Tor}(M,\mathfrak{g}_\rho,\lambda)$ given by 
\begin{align*}
&\frac{1}{m^{24} \left(m^2-1\right)^2}\Big(4 m^{34} (y-2) y+m^{32} \left(-19 y^2+42 y-4\right)+15 m^{30} \left(4 y^2-9 y+1\right)\\&
+m^{28} \left(-116 y^2+277 y-49\right)+m^{26} \left(138 y^2-343 y+59\right)+m^{24} \left(-113 y^2+286 y-49\right)\\&
+m^{22} \left(48 y^2-111 y+11\right)+m^{20} \left(-34 y^2+65 y-17\right)+2 m^{18} \left(2 y^2+5 y-20\right)\\&
+m^{16} \left(-12 y^2+85 y-113\right)-4 m^{14} \left(23 y^2-80 y+66\right)+m^{12} \left(-267 y^2+805 y-554\right)\\&
+m^{10} \left(-550 y^2+1633 y-1185\right)+m^8 \left(-1005 y^2+3242 y-2758\right)\\&
+m^6 \left(-2152 y^2+7471 y-6747\right)+m^4 \left(-5688 y^2+19271 y-16363\right)\\&
+m^2 (-4 y^{18}+117 y^{17}-1568 y^{16}+12790 y^{15}-71174 y^{14}+287221 y^{13}-872939 y^{12}\\&
\hspace{3em}+2051452 y^{11}-3807909 y^{10}+5701024 y^9-7049263 y^8+7379282 y^7-6653194 y^6\\&
\hspace{3em}+5176328 y^5-3442171 y^4+1908473 y^3-815463 y^2+230625 y-38691)\\&
+(-4 y^{16}+109 y^{15}-1354 y^{14}+10183 y^{13}-51960 y^{12}+191168 y^{11}-526550 y^{10}\\&
\hspace{3em}+1114704 y^9-1854044 y^8+2479718 y^7-2740522 y^6+2567766 y^5-2060060 y^4\\&
\hspace{3em}+1399565 y^3-784584 y^2+333247 y-77382) y\Big)
\end{align*}

We consider the case of $(p,q)=(1,1)$ and $x = 2 + 3 i$.
The system of $f_1(y,m)$, $g_1(y,m,l),$ and $m^pl^q-x$ has 17 zeros while the other system of $f_2(y,m)$, $g_2(y,m,l)$, and $m^pl^q-x$ has 20 zeros. Summing the inverse of the corresponding torsions, we numerically obtain $0.10320 + 0.00274 i$ from the first system and $-0.10320 - 0.00274 i$ from the second one. This verifies Conjecture \ref{conj:intro} numerically.

\begin{remark}\label{rmk:allcomponent}
The above computation shows that the sum in Conjecture \ref{conj:intro} should be considered over all components of $X^\mathrm{irr}(M)$, not only the component containing the geometric representation.
\end{remark}

\section{Further directions} \label{sec:further}

In this paper, we study the case  when there is a regular  defect  of maximal type along a knot $\mathcal{K}$.  Our study can be extended to the case  of a closed 3-manifold $\mathcal{M}$ without any defect. In the case, the corresponding SQFT $\mathcal{T}_{N}[\mathcal{M}]$ does not have any flavor symmetry and its twisted partition depends only on the discrete choice $g$ (genus of Riemann surface) and $N$. Let 
\begin{align*}
\begin{split}
\mathcal{I}_N (\mathcal{M},g) &:= (\textrm{Twisted index of $\mathcal{T}_N[\mathcal{M}]$ theory on $\Sigma_g \times S^1$})
\\
&= \textrm{Tr}_{\mathcal{H}^{E=0}_N(\mathcal{M}\times \Sigma_g)} (-1)^{\hat{R}}
\end{split}
\end{align*}
The corresponding 3D-3D relation was derived in \cite{Gang:2019uay,Benini:2019dyp}
\begin{align}
\mathcal{I}_N (\mathcal{M},g)  = \left( \dfrac{|\textrm{Hom}\left(\pi_1 (\mathcal{M}), \mathbb{Z}_N \right)|  } {N }\right)^{g-1} \mkern-20mu \sum_{[\rho] \in \frac{X^{\rm irr}_N(\mathcal{M})}{\textrm{Hom}\left(\pi_1 (\mathcal{M}), \mathbb{Z}_N \right)}} \mkern-20mu \left(\mathrm{Tor}(\mathcal{M},\mathfrak{g}_\rho) \right)^{g-1} \label{twistecd index for closed 3-manifold}
\end{align}
Here $X^{\rm irr}_N(\mathcal{M})$ denotes the set of irreducible $\mathrm{SL}_N(\mathbb{C})$ characters of $\pi_1(\mathcal{M})$. In the summation, two irreducible characters are considered to be equivalent if they are related to each other by tensoring a $\mathbb{Z}_N$ (center subgroup of $\mathrm{SL}_N(\mathbb{C})$) character. From the  argument in Section \ref{sec : no ground state}, we expect that $\textrm{dim} \mathcal{H}^{E=0}_N(\mathcal{M}\times \Sigma_{g=0})=0$ and thus
\begin{align*}
\mathcal{I}_N(\mathcal{M},g=0) =0.
\end{align*}

When $g=0$, there is another interesting property on $\Sigma_{g=0}$ which is absent for higher $g$. As $\Sigma_{g=0} = S^2$ admits $SO(3)$ isometry, we can introduce an Omega-deformation parameter, say $q$, on $\mathbb{B}= \Sigma_{g=0}\times S^1$ using the isometry. We can consider the refined twisted index for $g=0$ \cite{Benini:2015noa}
\begin{align*}
\mathcal{I}_N (\mathcal{M}, g=0 ;q) =  \textrm{Tr}_{\mathcal{H}^{E=0}_N(\mathcal{M}\times \Sigma_{g=0})} (-1)^{\hat{R}}q^{j_3}
\end{align*}
where $j_3$ is the Cartan of the $SO(3)$ isometry group on  $\Sigma_{g=0}$. The 3D-3D relation for the refined index is   \cite{Benini:2019dyp}
\begin{align*}
\begin{split}
&\mathcal{I}_N (\mathcal{M},g=0;q)  
\\
&= \left( \dfrac{|\textrm{Hom}\left(\pi_1 (\mathcal{M}), \mathbb{Z}_N \right)|  } {N }\right)^{-1} \mkern -30mu \sum_{[\rho] \in \frac{X^{\rm irr}_N(\mathcal{M})}{\textrm{Hom}\left(\pi_1 (\mathcal{M}), \mathbb{Z}_N \right)}} \mkern -23mu \exp \left( \sum_{n=0} 2 S_{2n+1}(\mathcal{M}, \mathfrak{g}_\rho) \hbar^{2n} \right)\bigg{|}_{\hbar := \log q}\;.
\end{split}
\end{align*}
Here $S_{k}(\mathcal{M}, \mathfrak{g}_\rho)$ is the $k$-loop perturbative invariant of $\mathrm{SL}_N(\mathbb{C})$ Chern-Simons theory on $\mathcal{M}$ around the flat-connection corresponding to the character $[\rho]$. In the unrefined case ($q=1$ or equivalently $\hbar =0$), only the 1-loop part $S_1 (\mathcal{M}, \mathfrak{g}_\rho) = -\frac{1}2 \log \textrm{Tor}(\mathcal{M}, \mathfrak{g}_\rho)$ contributes to the index. Since  $\textrm{dim} \mathcal{H}^{E=0}_N(\mathcal{M}\times \Sigma_{g=0})=0$, we have the following generalized vanishing conjecture for hyperbolic $\mathcal{M}$:
\begin{align*}
\sum_{[\rho] \in \frac{X^{\rm irr}_N(\mathcal{M})}{\textrm{Hom}\left(\pi_1 (\mathcal{M}), \mathbb{Z}_N \right)}} \mkern -20mu \exp \left( \sum_{n=0} 2 S_{2n+1}(\mathcal{M}, \mathfrak{g}_\rho) \hbar^{2n} \right) =0.
\end{align*}
Expanding the LHS as formal power series in $\hbar$, the conjecture says that   the series vanishes at every order in $\hbar$. Unfortunately, the perturbative invariants $S_{k}(\mathcal{M}, \mathfrak{g}_{\rho})$ for $k>1$ still lack a mathematical rigorous definition although there are some attempts  \cite{Bae:2016jpi,Gang:2017cwq} using state-integral models. 
 
Another  interesting future  direction is studying large $N$ limit of the twisted index in \eqref{twistecd index for closed 3-manifold}. The limit is interesting since the twisted index at large $N$ computes the supersymmetric microstates of supersymmetric black holes in anti-de-Sitter space-time.  
Refer to \cite{Gang:2018hjd,Gang:2019uay,Bae:2019poj} for studies in this direction. In \cite{Gang:2018hjd,Gang:2019uay}, it was assumed that the irreducible character $[\rho] = \rho_N \cdot [\rho_{\rm hyp}]$  and its complex conjugation give the most exponentially dominant contributions to the above summation in the large $N$ limit. Here $[\rho_{\rm hyp}]$ is the $\mathrm{SL}_2(\mathbb{C})$ character for the complete hyperbolic structure and $\rho_N :\mathrm{SL}_2(\mathbb{C})\rightarrow \mathrm{SL}_N(\mathbb{C})$ is the principal embedding. The asymptotic limit of the adjoint torsion twisted by   $\rho_N \cdot [\rho_{\rm hyp}]$ can be studied using the mathematical results in \cite{park2016reidemeister}. With the assumption and the mathematical result, the following large $N$ behavior is expected (for every hyperbolic $\mathcal{M}$ and $g>1$)
\begin{align*}
\lim_{N\rightarrow \infty } \frac{1}{N^3}\log |\mathcal{I}_N (\mathcal{M}, g  )|  = (g-1) \cdot \frac{\textrm{vol}(\mathcal{M})} {3\pi}.
\end{align*}
The result is compatible with the Bekenstein-Hawking entropy  of the corresponding AdS black holes \cite{Gang:2018hjd,Gang:2019uay,Bae:2019poj,Benini:2019dyp}. It would be interesting to prove or disprove the assumption used in the large $N$ analysis. 

\section*{Acknowledgments}
We are grateful to Teruaki Kitano, Francesco Benini, and Kazuya Yonekura for their helpful discussions. The researches of DG  were supported in part by the National Research Foundation of Korea  (NRF) grant 2019R1A2C2004880. The work of SK was supported by the National Research Foundation of Korea (NRF) grant funded by the Korea government (MSIT) (No. 2019R1C1C1003383).

%%%%%%%%%%%%%%%%%%%%%%%%%%%%%%%%%%%%%%%%%%%%%%%%%%%%
%%%%%%%%%%%%%%%%%%%%%  bibtex  %%%%%%%%%%%%%%%%%%%%%%%%%%%%%

% \bibliographystyle{nb}

\bibliographystyle{abbrv}
\bibliography{biblog}

\end{document}